\documentclass{amsart}
\usepackage{latexsym,amssymb,amsmath,amsthm,amscd,graphicx,ascmac,}
\usepackage{longtable}
\usepackage{enumerate}
\usepackage{color}

\numberwithin{equation}{section}

\theoremstyle{plain}
\newtheorem{thm}{Theorem}[section]
\newtheorem*{thm*}{Theorem}

\newtheorem{lem}[thm]{Lemma}
\newtheorem{prp}[thm]{Proposition}

\theoremstyle{definition}
\newtheorem{defn}[thm]{Definition}

\theoremstyle{remark}
\newtheorem{rem}[thm]{Remark}

\newcommand{\cA}{{\mathcal A}}

\newcommand{\cD}{{\mathcal D}}

\newcommand{\cM}{{\mathcal M}}

\newcommand{\cQ}{{\mathcal Q}}

\newcommand{\cS}{{\mathcal S}}

\newcommand{\N}{{\mathbb N}}

\newcommand{\R}{{\mathbb R}}

\newcommand{\Z}{{\mathbb Z}}

\def\al{\alpha}
\def\bt{\beta}
\def\gm{\gamma}

\def\Dl{\Delta}
\def\eps{\varepsilon}
\def\vphi{\varphi}

\def\lm{\lambda}

\def\sg{\sigma}

\def\tht{\theta}

\def\0{\emptyset}
\def\1{{\bf 1}}
\def\6{\partial}
\def\8{\infty}

\def\lt{\left}
\def\rt{\right}

\def\ds{\displaystyle}

\newcommand{\iii}[1]{{\left\vert\kern-0.25ex\left\vert\kern-0.25ex\left\vert #1 
    \right\vert\kern-0.25ex\right\vert\kern-0.25ex\right\vert}}

\def\Xint#1{\mathchoice
 {\XXint\displaystyle\textstyle{#1}}%
 {\XXint\textstyle\scriptstyle{#1}}%
 {\XXint\scriptstyle\scriptscriptstyle{#1}}%
 {\XXint\scriptscriptstyle\scriptscriptstyle{#1}}%
 \!\int}
 \def\XXint#1#2#3{{\setbox0=\hbox{$#1{#2#3}{\int}$}
 \vcenter{\hbox{$#2#3$}}\kern-.5\wd0}}

\def\fint{\Xint-}

\begin{document}

\title{Multilinear embedding theorem for fractional sparse operators}

\author[N.~Hatano]{Naoya Hatano}
\address{
Graduate School of Information Science and Technology, The University of Osaka, 1-5, Yamadaoka, Suita-shi, Osaka 565-0871, Japan
}
\email{n.hatano.chuo@gmail.com}

\author[R.~Kawasumi]{Ryota Kawasumi}
\address{
Center for Mathematics and Data Science, Gunma University, 
4-2 Aramaki-machi, Maebashi City, Gunma 371-8510, Japan 
}
\email{r-kawasumi@gunma-u.ac.jp}

\author[H.~Saito]{Hiroki Saito}
\address{
College of Science and Technology, Nihon University,
Narashinodai 7-24-1, Funabashi City, Chiba, 274-8501, Japan
}
\email{saitou.hiroki@nihon-u.ac.jp}

\author[H.~Tanaka]{Hitoshi Tanaka}
\address{
Research and Support Center on Higher Education for the hearing and Visually Impaired, 
National University Corporation Tsukuba University of Technology,
Kasuga 4-12-7, Tsukuba City, Ibaraki, 305-8521 Japan
}
\email{htanaka@k.tsukuba-tech.ac.jp}

\thanks{
The first named author is financially supported by 
the Grant-in-Aid for JSPS Fellows (No. 25KJ0222).
The second named author is supported by 
Grant-in-Aid for Research Activity Start-up (25K23330)
The third named author is supported by 
Grant-in-Aid for Scientific Research (C) (23K03171), 
the Japan Society for the Promotion of Science. 
The fourth named author is supported by 
Grant-in-Aid for Scientific Research (C) (15K04918, 19K03510 and 26K06821), 
the Japan Society for the Promotion of Science.
}



\date{}

\begin{abstract}
We show some simple sufficient conditions under which 
the multilinear embedding theorem holds for fractional sparse operators.
By verifying these conditions, 
we establish the theorem for power weights.
We also clarify the relation between 
vanishing local Morrey spaces and 
$L^p \to L^q$, $1<p,q<\8$, 
infinitesimal relative boundedness theorem.
More precisely,
we establish that 
belonging to some vanishing local Morrey spaces 
of the weight $v$ with an appropriate power
provides sufficient conditions under which
$L^p \to L^q$, $1<p,q<\8$, 
infinitesimal relative bounds holds 
for Schr\"{o}dinger operators 
$(-\Dl)^{\al/2}+v$. 
\end{abstract}

\maketitle

\section{Introduction}\label{sec1}
Sparse operators play an important role 
in harmonic analysis 
(\cite{Co,Cr,CrMo,Hy2,Hy,Lec,Le,LeOm}).
In \cite{Le0}, 
as a~simple proof of the $A_2$ conjecture,
A.K.~Lerner established 
sparse domination for the singular integral operators.
Probably, one of his 
significant contributions is that 
he clearly separated 
the analysis of singular integral operators 
(that is, sparse domination for the singular integral operators) 
and the weighted estimates of sparse operators.
Once an operator is controlled by a~sparse operator,
we can focus on the properties of the weights 
to prove the weighted estimates for a~sparse operator.
This makes it possible to treat many operators
in a~simple and unified way.

The purpose of this paper is to study 
the weighted multilinear embedding inequality:
\begin{equation}\label{1.1}
\sum_{S\in\cS}
K(S)\prod_{i=1}^n
\int_{S}f_i\,{\rm d}\sg_i
\lesssim A_0
\prod_{i=1}^n
\|f_i\|_{L^{p_i}(\sg_i)},
\end{equation}
where $n\ge 2$,
$\sg_1,\ldots,\sg_n$ 
are weights on $\R^d$,
$\cS$ is a~sparse family,
$K:\cS\to[0,\8)$ 
is a~map, 
$1<p_1,\ldots,p_n<\8$ 
are appropriate exponents and 
$f_1,\ldots,f_n\in L_{{\rm loc}}^1({\rm d}x)$
are the nonnegative functions.
More precisely, 
we study some simple sufficient conditions 
for which $A_0<\8$.
By a~duality argument, 
one knows that 
the weighted multilinear embedding inequality \eqref{1.1} 
is equivalent to
the weighted multilinear norm inequality:
\[
\|\cA_{\cS,K}[f_1\sg_1,\ldots,f_{n-1}\sg_{n-1}]\|_{L^{p_n'}(\sg_n)}
\lesssim A_0
\prod_{i=1}^{n-1}
\|f_i\|_{L^{p_i}(\sg_i)},
\]
where
$p_{n}'=\frac{p_n}{p_n-1}$,
and $\cA_{\cS,K}$ 
is a weighted multilinear sparse operator 
defined by
\[
\cA_{\cS,K}[f_1\sg_1,\ldots,f_{n-1}\sg_{n-1}](x)
:=
\sum_{S\in\cS}
K(S)\prod_{i=1}^{n-1}
\int_{S}|f_i|\,{\rm d}\sg_i\,\1_{S}(x),
\quad x\in\R^d,
\]
where $\1_{E}$ denotes the characteristic function of the set $E$. 
In particular, 
we construct a sparse operator associated with the Bessel potential 
and apply it to establish a simple sufficient condition 
for the infinitesimal relative boundedness 
of certain perturbations of 
the fractional Schr\"{o}dinger operator.
This simple sufficient condition 
enables us to clarify the relation 
between vanishing local Morrey spaces and 
the $L^p \to L^q$, $1<p,q<\8$, 
infinitesimal relative boundedness theorem 
(see Remark \ref{rem1.1} below).
For further details on sparse domination of singular integrals
and Riesz potentials,
see the papers \cite{CrMo,Le} and 
the surveys \cite{Cr,Du,Pe}.

We first summarize the notation used in this paper.
Denote by $\cQ=\cQ(\R^d)$ 
the family of all cubes in $\R^d$ 
with sides parallel to the axes. 
Given a~cube $Q\in\cQ$, 
denote by $c_{Q}$ and $\ell_{Q}$ 
its center and its side length of $Q$,
respectively. 
$|Q|$ stands for the volume of $Q$.
Denote by $\cD$ the family of all dyadic cubes 
\[
\cD
:=
\{2^{-k}(m+[0,1)^d):\,
k\in\Z,\,m\in\Z^d\}.
\]
Let $0<\eta<1$.
We say that a family 
$\cS\subset\cD$ is \textit{$\eta$-sparse}
if for every $S\in\cS$, 
there exists a measurable set 
$E_{\cS}(S)\subset S$ such that 
$|E_{\cS}(S)|\ge\eta|S|$, 
and the sets 
$\{E_{\cS}(S):\,S\in\cS\}$ 
are pairwise disjoint. 
Given $1<p<\8$, $p'=\frac{p}{p-1}$ 
denotes the conjugate exponent number of $p$. 
Recall that a weight 
(that is, a nonnegative locally integrable function) 
$w$ satisfies the $A_p$, $1<p<\8$, condition if
\[
[w]_{A_p}
:=
\sup_{Q\in\cQ}
\lt(\frac1{|Q|}\int_{Q}w\,{\rm d}x\rt)
\lt(\frac1{|Q|}\int_{Q}w^{-\frac1{p-1}}\,{\rm d}x\rt)^{p-1}<\8,
\]
and 
\[
[w]_{A_{\8}}
:=
\lim_{p\to\8}[w]_{A_p}<\8.
\]

For the theory of weighted norm inequalities,
we refer the reader to the classical text \cite{GR} as well as \cite{Gr}.
Regarding the $n$-linear embedding theorem for dyadic positive operators 
and related topics of dyadic analysis,
see our recent book \cite{ST}.

The letter $C$ will be used for constants that may change from one occurrence to another.
Constants with subscripts, such as $C_1$, $C_2$, do not change in different occurrences.
By $A\approx B$ we mean that 
$c^{-1}B\le A\le cB$ 
with some positive finite constant $c$ independent of appropriate quantities. 
We write $X\lesssim Y$, $Y\gtrsim X$ 
if there is a independent constant $c$ such that $X \le cY$. 

\subsection*{Theorems}
We will state our main results of this paper.

\begin{thm}\label{thm2.1}
Let $\sg_1,\ldots,\sg_n$ 
be weights on $\R^d$ and let 
$K:\cD\to[0,\8)$ be a~map. 
Let $\cS\subset\cD$ be an~$\eta$-sparse family.
Let $1<p_1,\ldots,p_n<\8$ 
with $1/p_1+\cdots+1/p_n\ge 1$.
Then, for the nonnegative functions 
$f_1,\ldots,f_n\in L_{{\rm loc}}^1({\rm d}x)$, 
the multilinear embedding inequality
\[
\sum_{S\in\cS}
K(S)\prod_{i=1}^n
\int_{S}f_i\,{\rm d}\sg_i
\lesssim A_0
\prod_{i=1}^n
\|f_i\|_{L^{p_i}(\sg_i)}
\]
holds if 
the following assertions are satisfied.
\begin{itemize}
\item[{\rm(a)}]
If 
$\sg_1,\ldots,\sg_n\in A_{\8}$,
\[
A_0
=
\sup_{S\in\cS}
K(S)|S|^{n-1/p_1-\cdots-1/p_n}
\prod_{i=1}^n
\lt(\frac{\sg_i(S)}{|S|}\rt)^{1/p_i'}
<\8.
\]
\item[{\rm(b)}]
If not, for $\tht>1$,
\[
A_0
=
\sup_{S\in\cS}
K(S)|S|^{n-1/p_1-\cdots-1/p_n}
\prod_{i=1}^n
\lt(\frac{\sg_i^{\tht}(S)}{|S|}\rt)^{\frac1{\tht p_i'}}
<\8.
\]
\end{itemize}
\end{thm}

\begin{thm}\label{thm2.2}
Let $\sg_1,\ldots,\sg_n$ 
be weights on $\R^d$ and let 
$K:\cD\to[0,\8)$ be a~map. 
Let $\cS\subset\cD$ be an~$\eta$-sparse family.
Let $1<p_1,\ldots,p_n<\8$ 
with $1/p_1+\cdots+1/p_n<1$.
Then, for the nonnegative functions 
$f_1,\ldots,f_n\in L_{{\rm loc}}^1({\rm d}x)$, 
the multilinear embedding inequality
\[
\sum_{S\in\cS}
K(S)\prod_{i=1}^n
\int_{S}f_i\,{\rm d}\sg_i
\lesssim A_0
\prod_{i=1}^n
\|f_i\|_{L^{p_i}(\sg_i)}
\]
holds if 
the following assertions are satisfied.
\begin{itemize}
\item[{\rm(a)}]
If 
$\sg_1,\ldots,\sg_n\in A_{\8}$,
for $1/r+1/p_1+\cdots+1/p_n=1$,
\[
A_0
=
\lt[
\sum_{S\in\cS}
\lt(
K(S)|S|^{n-1}\prod_{i=1}^n
\lt(\frac{\sg_i(S)}{|S|}\rt)^{1/p_i'}
\rt)^r
|E_{\cS}(S)|
\rt]^{1/r}
<\8.
\]
\item[{\rm(b)}]
If not, 
for $1/r+1/p_1+\cdots+1/p_n=1$ 
and $\tht>1$,
\[
A_0
=
\lt[
\sum_{S\in\cS}
\lt(
K(S)|S|^{n-1}\prod_{i=1}^n
\lt(\frac{\sg_i^{\tht}(S)}{|S|}\rt)^{\frac1{\tht p_i'}}
\rt)^r
|E_{\cS}(S)|
\rt]^{1/r}
<\8.
\]
\end{itemize}
\end{thm}

When $f_1,\ldots,f_n$ are 
signed functions, we can use 
Theorems \ref{thm2.1} and \ref{thm2.2}
by replacing them with
$|f_1|,\ldots,|f_n|$.

Condition (b) in Theorems \ref{thm2.1} and \ref{thm2.2} 
corresponds to a multilinear variant of 
the classical power-type bump condition 
initiated by Neugebauer \cite{Ne}. 
While modern two-weight theory has developed more subtle bump conditions,
the power bump condition provides a simple, 
direct, and tractable criterion that 
effectively handles non-$A_{\8}$ weights 
in the multilinear fractional setting.
Aiming at applications to vanishing local Morrey spaces, 
we intentionally adopt the power bump condition 
rather than general Orlicz--Morrey settings 
to avoid unnecessary technical complications.

We give several applications of
Theorems \ref{thm2.1} and \ref{thm2.2}.
First, by applying part (a) of 
Theorems \ref{thm2.1} and \ref{thm2.2}
to the homogeneous and inhomogeneous-type power weights, 
respectively,
we obtain the following results.

\begin{thm}\label{thm3}
Let 
$K(S)=\ell_{S}^{\al-(n-1)d}$, 
$S\in\cS$,\footnote{
The exponent $\al-(n-1)d$ 
reflects the fact that 
the underlying Riesz potential 
is $(n-1)$-linear, 
with the $n$-th function serving as the output test function.
}
with $0<\al<(n-1)d$.
\begin{itemize}
\item[{\rm(a)}]
Suppose that $1/p_1+\cdots+1/p_n\ge 1$.
Let $\sg_i=|x|^{\bt_i}$ with 
$\bt_i>-d$, $i=1,\ldots,n$.
Then, the constant $A_0$ in 
Theorem \ref{thm2.1} part (a) 
is finite provided that
\[
\al
+
d\lt(1-\frac{1}{p_1}-\cdots-\frac{1}{p_n}\rt)
\ge 0
\quad\text{and}\quad
\al
+
d\lt(1-\frac{1}{p_1}-\cdots-\frac{1}{p_n}\rt)
+
\frac{\bt_1}{p_1'}+\cdots+\frac{\bt_n}{p_n'}
=0.
\]
\item[{\rm(b)}]
Suppose that $1/p_1+\cdots+1/p_n< 1$.
Let 
$\sg_i
=
|x|^{\bt_i}
\max(|x|,\,1)^{\frac{p_i'\rho}{n}
}$ 
with $\rho<-d/r$ and 
$\bt_i+p_i'\rho/n>-d$, 
$i=1,\ldots,n$.
Then, the constant $A_0$ in 
Theorem \ref{thm2.2} part (a) 
is finite provided that
\[
\al
+
\frac{\bt_1}{p_1'}+\cdots+\frac{\bt_n}{p_n'}
=0.
\]
\end{itemize}
\end{thm}

This result extends 
the classical Stein-Weiss inequality 
\cite{StWe}, 
since 
Theorem \ref{thm3} part (a) 
reduces to the classical case 
when $n=2$.
Further details will be described in Section 3.

As a~second application of Theorems \ref{thm2.1} and \ref{thm2.2},
we derive a~simple sufficient condition of 
the $L^p \to L^q$ infinitesimal relative bounds 
for Schr\"{o}dinger operators of the form
$(-\Dl)^{\al/2}+v$
(Theorems \ref{thm4.1} and \ref{thm4.2}).

\begin{thm}\label{thm4.1}
Let $0<\al<d$, $1<p\le q<\8$ 
and $v$ be a~weight on $\R^d$. Then
for any $\eps>0$, there exists 
$C(\eps)>0$ such that, 
for any $\vphi\in C_c^{\8}(\R^d)$,
\[
\|\vphi\|_{L^q(v)}
\lesssim
\eps
\|(-\Dl)^{\al/2}\vphi\|_{L^p({\rm d}x)}
+
C(\eps)
\|\vphi\|_{L^p({\rm d}x)}
\]
holds if 
the following assertions are satisfied.
\begin{itemize}
\item[{\rm(a)}]
If $v\in A_{\8}$,
\[
\sup_{Q\in\cQ(\R^d):\,
\ell_{Q}\le 1}
|Q|^{\al/d-1/p}v(Q)^{1/q}
<\8
\]
and
\[
\lim_{\lm\to\8}
\sup_{Q\in\cQ(\R^d):\,
\ell_{Q}\le 1/\lm}
|Q|^{\al/d-1/p}v(Q)^{1/q}
=0.
\]
\item[{\rm(b)}]
If $v\notin A_{\8}$, 
for $\tht>1$,
\[
\sup_{Q\in\cQ(\R^d):\,
\ell_{Q}\le 1}
|Q|^{\al/d+1/q-1/p}
\lt(\frac{v^{\tht}(Q)}{|Q|}\rt)^{\frac1{\tht q}}
<\8
\]
and
\[
\lim_{\lm\to\8}
\sup_{Q\in\cQ(\R^d):\,
\ell_{Q}\le 1/\lm}
|Q|^{\al/d+1/q-1/p}
\lt(\frac{v^{\tht}(Q)}{|Q|}\rt)^{\frac1{\tht q}}
=0.
\]
\end{itemize}
\end{thm}

\begin{thm}\label{thm4.2}
Let $0<\al<d$, $1<q<p<\8$ 
and $v$ be a~weight on $\R^d$. 
Let $1/r=1/q-1/p$ and 
$\tht_1,\tht_2>1$ with
\[
\frac1q
=
\frac1{\tht_2 p}
+
\frac{\tht_1}r.
\]
Suppose that 
$\|v\|_{L^{\tht_1}({\rm d}x)}^{\tht_1/r}<\8$.
Then for any $\eps>0$, there exists 
$C(\eps)>0$ such that, 
for any $\vphi\in C_c^{\8}(\R^d)$,
\[
\|\vphi\|_{L^q(v)}
\lesssim
\eps
\|(-\Dl)^{\al/2}\vphi\|_{L^p({\rm d}x)}
+
C(\eps)
\|\vphi\|_{L^p({\rm d}x)}
\]
holds if 
the following assertions are satisfied.
\begin{itemize}
\item[{\rm(a)}]
If $v\in A_{\8}$,
\[
\sup_{Q\in\cQ(\R^d):\,
\ell_{Q}\le 1}
\ell_{Q}^{\al}
\lt(\frac{v(Q)}{|Q|}\rt)^{\frac1{\tht_2 p}}
<\8
\]
and
\[
\lim_{\lm\to\8}
\sup_{Q\in\cQ(\R^d):\,
\ell_{Q}\le 1/\lm}
\ell_{Q}^{\al}
\lt(\frac{v(Q)}{|Q|}\rt)^{\frac1{\tht_2 p}}
=0.
\]
\item[{\rm(b)}]
If $v\notin A_{\8}$,
for $\tht\in(1, \tht_1)$,
\[
\sup_{Q\in\cQ(\R^d):\,
\ell_{Q}\le 1}
\ell_{Q}^{\al}
\lt(\frac{v^{\tht}(Q)}{|Q|}\rt)^{\frac1{\tht\tht_2 p}}
<\8
\]
and
\[
\lim_{\lm\to\8}
\sup_{Q\in\cQ(\R^d):\,
\ell_{Q}\le 1/\lm}
\ell_{Q}^{\al}
\lt(\frac{v^{\tht}(Q)}{|Q|}\rt)^{\frac1{\tht\tht_2 p}}
=0.
\]
\end{itemize}
\end{thm}

\begin{rem}\label{rem1.1}
Let $1<p<\8$ and $0\le s<d$. 
A~weight $w$ is said to belong to 
the \textit{vanishing local Morrey space}
${\rm vl}\cM_p^s(\R^d)$ if\footnote{
For details on Morrey spaces and 
their application to Schr\"{o}dinger operators,
see e.g., \cite{SDH20,SDH20-2}.
}
\[
\|w\|_{{\rm vl}\cM_p^s}
:=
\sup_{Q\in\cQ(\R^d):\,
\ell_{Q}\le 1}
|Q|^{\frac{s}d-\frac1p}
\|w\1_{Q}\|_{L^p({\rm d}x)}
<\8
\]
and
\[
\lim_{\lm\to\8}
\sup_{Q\in\cQ(\R^d):\,
\ell_{Q}\le 1/\lm}
|Q|^{\frac{s}d-\frac1p}
\|w\1_{Q}\|_{L^p({\rm d}x)}
=0.
\]
Then:
\begin{itemize}
\item[$\bullet$] 
the assertion (a) in 
Theorem \ref{thm4.1} means that 
$v^{1/q}\in
{\rm vl}\cM_q^{\al+d/q-d/p}(\R^d)$;
\item[$\bullet$] 
the assertion (b) in 
Theorem \ref{thm4.1} means that, 
for $r>q$, 
$v^{1/q}\in
{\rm vl}\cM_r^{\al+d/q-d/p}(\R^d)$;
\item[$\bullet$] 
the assertion (a) in 
Theorem \ref{thm4.2} means that 
$v^{1/(\tht_2p)}\in
{\rm vl}\cM_{\tht_2p}^{\al}(\R^d)$;
\item[$\bullet$] 
the assertion (b) in 
Theorem \ref{thm4.2} means that, 
for $r>p$, 
$v^{1/(\tht_2p)}\in
{\rm vl}\cM_{\tht_2r}^{\al}(\R^d)$.
\end{itemize}
Thus, 
belonging to some vanishing local Morrey spaces 
of the weight $v$ with an appropriate power 
provides sufficient conditions under which
$L^p \to L^q$, $1<p,q<\8$, 
infinitesimal relative bounds hold
for Schr\"{o}dinger operators 
$(-\Dl)^{\al/2}+v$. 
\end{rem}

\begin{rem}\label{rem1.2}
One naively takes 
$C(\eps)=\|v\|_{L^{\8}({\rm d}x)}$
which might be infinity.
But, 
thanks to the smoothness of $\vphi$, 
one has $C(\eps)<\8$ 
when $v$ satisfies 
the vanishing Morrey property.
Theorem \ref{thm4.1}, 
for the case $p=q$ and 
$v$ be a~power weight, 
was first due to \cite{CDJ}.
In \cite{HKST}, 
we gave some Morrey-type sufficient conditions for which
$L^p \to L^q$, $1<p,q<\8$, 
infinitesimal relative bounds holds 
for Schr\"{o}dinger operators 
$(-\Dl)^{\al/2}+v$ 
in terms of bilinear embedding theorems for dyadic positive operators.
We also refer the reader to Maz'ya and Verbitsky \cite{MaVe} 
for fundamental criteria on infinitesimal form boundedness 
in the classical $L^2$ setting;
see also \cite{Ch}.
In \cite{MaVe},
necessary and sufficient conditions for
the infinitesimal form boundedness were completely established
in terms of Sobolev capacities and dyadic Carleson measure conditions,
which also characterize the classical Kato class and local 
Fefferman--Phong type conditions.
In comparison,
our approach via fractional sparse domination and multilinear embedding theorems
allows us to treat general fractional orders
$0<\al<d$ and full exponent ranges
$1<p,q<\8$ in a unified framework,
providing verifiable sufficient conditions in terms of
vanishing local Morrey spaces even for non-$A_{\8}$ weights.
\end{rem}

\section{Multilinear embedding theorems}\label{sec2}
In what follows we prove Theorems 
\ref{thm2.1} and \ref{thm2.2}.

\subsection{Proof of Theorem \ref{thm2.1}}\label{ssec2.1}
Let $1<p_1,\ldots,p_n<\8$ 
with $1/p_1+\cdots+1/p_n\ge 1$.
For the nonnegative functions 
$f_1,\ldots,f_n\in L_{{\rm loc}}^1({\rm d}x)$, 
we analyze the following multilinear embedding inequality:
\begin{equation}\label{2.1}
\sum_{S\in\cS}
K(S)\prod_{i=1}^n
\int_{S}f_i\,{\rm d}\sg_i
\lesssim C_1
\prod_{i=1}^n
\|f_i\|_{L^{p_i}(\sg_i)}.
\end{equation}
Letting
\[
C_1
:=
\sup_{S\in\cS}
K(S)\prod_{i=1}^n
\frac{\sg_i(S)}{\sg_i(E_{\cS}(S))^{1/p_i}},
\]
using multiple H\"{o}lder's inequality\footnote{
Letting
$\tht=1/p_1+\cdots+1/p_n$,
using multiple H\"{o}lder's inequality
with the exponents $\tht p_i$,
namely,
\[
\sum_{S\in\cS}
\prod_{i=1}^{n}
a_{i,S}
\le
\prod_{i=1}^{n}
\lt(
\sum_{S\in\cS}a_{i,S}^{\theta p_i}
\rt)^{1/(\theta p_i)},
\]
and
we apply the norm inequality
$\|\quad\|_{l^{p_i}}\ge\|\quad\|_{l^{\tht p_i}}$.
}
and using the boundedness of dyadic Hardy-Littlewood maximal operators,
we have that
\begin{align*}
&\sum_{S\in\cS}
K(S)\prod_{i=1}^n
\int_{S}f_i\,{\rm d}\sg_i
\\ &\le C_1
\sum_{S\in\cS}\prod_{i=1}^n
\fint_{S}f_i\,{\rm d}\sg_i
\cdot
\sg_i(E_{\cS}(S))^{1/p_i}
\\ &\le C_1
\prod_{i=1}^n
\lt[
\sum_{S\in\cS}
\lt(\fint_{S}f_i\,{\rm d}\sg_i\rt)^{p_i}
\sg_i(E_{\cS}(S))
\rt]^{1/p_i}
\\ &\lesssim C_1
\prod_{i=1}^n
\|f_i\|_{L^{p_i}(\sg_i)},
\end{align*}
where $\fint_{S}f\,{\rm d}\sg$ 
denotes the integral average 
$\sg(S)^{-1}\int_{S}f\,{\rm d}\sg$.

\begin{rem}\label{rem2.1}
We notice that,
if the weight $w\in A_{\8}$,
then, for some $p_0>1$,
the Hardy-Littlewood maximal operator $M$
is bounded on $L^{p_0}(w)$.
This implies that
\[
w(S)
\le \eta^{-p_0}
\lt(\frac{|E_{\cS}(S)|}{|S|}\rt)^{p_0}
w(S)
\le \eta^{-p_0}
\int_{S}
\lt(M\1_{E_{\cS}(S)}\rt)^{p_0}
\,{\rm d}w
\lesssim
\eta^{-p_0}
[w]_{A_{p_0}}^{\max(p_0,p_0')}
w(E_{\cS}(S)).
\]
This is the reason why
the statement (a) is included,
even though the statement (a) follows from
the statement (b) by the reverse H\"older property of $A_{\8}$ weights.
\end{rem}

If we assume 
$\sg_1,\ldots,\sg_n\in A_{\8}$,
and apply Remark \ref{rem2.1},
then we obtain
\begin{align*}
C_1
&\lesssim
\sup_{S\in\cS}
K(S)
\prod_{i=1}^n\sg_i(S)^{1/p_i'}
\\ &=
\sup_{S\in\cS}
K(S)|S|^{n-1/p_1-\cdots-1/p_n}
\prod_{i=1}^n
\lt(\frac{\sg_i(S)}{|S|}\rt)^{1/p_i'},
\end{align*}
which proves Theorem \ref{thm2.1} 
part (a).
Here, we have used
\[
n-1/p_1-\cdots-1/p_n
=
1/p_1'+\cdots+1/p_n'.
\]

Rewrite
\begin{equation}\label{2.2}
\sum_{S\in\cS}
K(S)\prod_{i=1}^n
\int_{S}f_i\,{\rm d}\sg_i
=
\sum_{S\in\cS}
K(S)|S|^n\prod_{i=1}^n
\fint_{S}f_i\sg_i\,{\rm d}x.
\end{equation}
Let $\tht>1$ and set 
$q_i'=\tht p_i'$. 
We notice that
\begin{equation}\label{2.3}
\fint_{S}f_i\sg_i\,{\rm d}x
\le
\lt(\fint_{S}(\sg_i^{1/p_i'})^{q_i'}\,{\rm d}x\rt)^{1/q_i'}
\lt(\fint_{S}(f_i\sg_i^{1/p_i})^{q_i}\,{\rm d}x\rt)^{1/q_i}.
\end{equation}
Thanks to this quite simple inequality,
we let
\[
C_2
:=
\sup_{S\in\cS}
K(S)|S|^{n-1/p_1-\cdots-1/p_n}
\prod_{i=1}^n
\lt(\fint_{S}(\sg_i^{1/p_i'})^{q_i'}\,{\rm d}x\rt)^{1/q_i'},
\]
apply the multiple H\"{o}lder's inequality,
notice that
$p_i>q_i$ since 
$q_i'=\theta p_i'>p_i'$, and 
use the boundedness of dyadic Hardy-Littlewood maximal operators
to conclude that
\begin{align*}
&\sum_{S\in\cS}
K(S)|S|^n\prod_{i=1}^n
\fint_{S}f_i\sg_i\,{\rm d}x
\\ &\le C_2
\sum_{S\in\cS}\prod_{i=1}^n
\lt(\fint_{S}(f_i\sg_i^{1/p_i})^{q_i}\,{\rm d}x\rt)^{1/q_i}
|S|^{1/p_i}
\\ &\lesssim C_2
\prod_{i=1}^n
\lt[
\sum_{S\in\cS}
\lt(\fint_{S}(f_i\sg_i^{1/p_i})^{q_i}\,{\rm d}x\rt)^{p_i/q_i}
|E_{\cS}(S)|
\rt]^{1/p_i}
\\ &\le C_2
\prod_{i=1}^n
\lt(
\int_{\R^d}
M[(f_i\sg_i^{1/p_i})^{q_i}](x)^{p_i/q_i}
\,{\rm d}x
\rt)^{1/p_i}
\\ &\lesssim C_2
\prod_{i=1}^n
\|f_i\|_{L^{p_i}(\sg_i)}.
\end{align*}
There holds
\begin{equation}\label{2.4}
K(S)|S|^{n-1/p_1-\cdots-1/p_n}
\prod_{i=1}^n
\lt(\fint_{S}(\sg_i^{1/p_i'})^{q_i'}\,{\rm d}x\rt)^{1/q_i'}
=
K(S)|S|^{n-1/p_1-\cdots-1/p_n}
\prod_{i=1}^n
\lt(\frac{\sg_i^{\tht}(S)}{|S|}\rt)^{\frac1{\tht p_i'}},
\end{equation}
which proves Theorem \ref{thm2.1} 
part (b).
\qed

\subsection{Proof of Theorem \ref{thm2.2}}\label{ssec2.2}
Let $1<p_1,\ldots,p_n<\8$ 
with $1/p_1+\cdots+1/p_n<1$.
Letting 
$1/r+1/p_1+\cdots+1/p_n=1$,
we have that
\begin{align*}
&\sum_{S\in\cS}
K(S)\prod_{i=1}^n
\int_{S}f_i\,{\rm d}\sg_i
\\ &=
\sum_{S\in\cS}
K(S)\prod_{i=1}^n
\frac{\sg_i(S)}{\sg_i(E_{\cS}(S))^{1/p_i}}
\cdot
\prod_{i=1}^n
\fint_{S}f_i\,{\rm d}\sg_i
\cdot
\sg_i(E_{\cS}(S))^{1/p_i}
\\ &\le
\lt[
\sum_{S\in\cS}
\lt(
K(S)\prod_{i=1}^n
\frac{\sg_i(S)}{\sg_i(E_{\cS}(S))^{1/p_i}}
\rt)^r\rt]^{1/r}
\prod_{i=1}^n
\lt[
\sum_{S\in\cS}
\lt(\fint_{S}f_i\,{\rm d}\sg_i\rt)^{p_i}
\sg_i(E_{\cS}(S))
\rt]^{1/p_i}.
\end{align*}
By letting
\[
C_3
=
\lt[
\sum_{S\in\cS}
\lt(
K(S)\prod_{i=1}^n
\frac{\sg_i(S)}{\sg_i(E_{\cS}(S))^{1/p_i}}
\rt)^r\rt]^{1/r},
\]
this yields
\[
\sum_{S\in\cS}
K(S)\prod_{i=1}^n
\int_{S}f_i\,{\rm d}\sg_i
\lesssim C_3
\prod_{i=1}^n
\|f_i\|_{L^{p_i}(\sg_i)},
\]
where we have used the
boundedness of 
the weighted dyadic Hardy-Littlewood  maximal operators 
$M_{\sg_{i}}$ on 
$L^{p_{i}}(\sg_{i})$.
If we assume 
$\sg_1,\ldots,\sg_n\in A_{\8}$
and apply Remark \ref{rem2.1},
then we obtain
\begin{align*}
C_3
&\lesssim
\lt[
\sum_{S\in\cS}
\lt(
K(S)
\prod_{i=1}^n\sg_i(S)^{1/p_i'}
\rt)^r\rt]^{1/r}
\\ &=
\lt[
\sum_{S\in\cS}
\lt(
K(S)|S|^{n-1}\prod_{i=1}^n
\lt(\frac{\sg_i(S)}{|S|}\rt)^{1/p_i'}
\rt)^r
|E_{\cS}(S)|
\rt]^{1/r},
\end{align*}
which proves Theorem \ref{thm2.2} 
part (a).
Here, we have used
\[
n+1/r
=
1+1/p_1'+\cdots+1/p_n'.
\]

Theorem \ref{thm2.2} 
part (b) follows from
\eqref{2.2}--\eqref{2.4} and
\[
|S|
=
|S|^{1/r}
\prod_{i=1}^n|S|^{1/p_i}.
\qed
\]

\section{Application to 
homogeneous and inhomogeneous-type power weights
}\label{sec3}
In what follows we prove Theorem \ref{thm3}.
We set $K(S)=\ell_{S}^{\al-(n-1)d}$, $S\in\cS$.
In this case,
the operator $\cA_{\cS,K}$, 
as defined in the Introduction,
is the sparse operator 
associated with the Riesz potential 
(cf.~\cite{Cr}).
We use the following simple and nice lemma.

\begin{lem}[{\rm\cite[Example 113]{SDH20}}]\label{lem3.1}
Let $Q\in\cQ(\R^d)$.
If $\bt>-d$, then
\[
\int_{Q}|x|^{\bt}\,{\rm d}x
\approx
\max(\ell_{Q},\,|c_{Q}|)^{\bt}|Q|.
\]
\end{lem}

\subsection{
The case $1/p_1+\cdots+1/p_n\ge 1$
}
Let
$\sg_i=|x|^{\bt_i}$ with $\bt_i>-d$.
Then $\sg_i\in A_{\8}$.
We examine the condition in 
Theorem \ref{thm2.1} part (a):
noticing that 
$\al+d-d/p_1-\cdots-d/p_n\ge 0$,
\begin{align*}
&K(S)|S|^{n-1/p_1-\cdots-1/p_n}\prod_{i=1}^n
\lt(\frac{\sg_i(S)}{|S|}\rt)^{1/p_i'}
\\ &\approx
\ell_{S}^{\al+d-d/p_1-\cdots-d/p_n}
\max(\ell_{S},\,|c_{S}|)^{
\bt_1/p_1'+\cdots+\bt_n/p_n'
}\\ &\le
\max(\ell_{S},\,|c_{S}|)^{
\al+d-d/p_1-\cdots-d/p_n
}
\max(\ell_{S},\,|c_{S}|)^{
\bt_1/p_1'+\cdots+\bt_n/p_n'
}
\\ &=
\max(\ell_{S},\,|c_{S}|)^{
\al+d(1-1/p_1-\cdots-1/p_n)
+
\bt_1/p_1'+\cdots+\bt_n/p_n'
}=1,
\end{align*}
by the assumption that
$\al
+
d(1-1/p_1-\cdots-1/p_n)
+
\bt_1/p_1'+\cdots+\bt_n/p_n'=0$.

\begin{rem}\label{rem3.1}
If we set $n=2$, 
$p_2'=q$, and $p_1=p$
and define 
$\gm_i:=\bt_i/p_i'$,
then applying 
Theorem \ref{thm3} part (a),
we obtain the following inequality,
which is known as 
the Stein-Weiss inequality:
\[
\||x|^{\gm_2}I_{\al}f(x)\|_{L^q(\R^d)}
\lesssim
\||x|^{-\gm_1}f(x)\|_{L^p(\R^d)},
\]
provided that
$
\frac{1}{q}
=
\frac{1}{p}
-\frac{\al-\gm_1-\gm_2}d
$
and
$\gm_1+\gm_2\le0$.
\end{rem}

\subsection{
The case $1/p_1+\cdots+1/p_n<1$
}
Let
$\sg_i=|x|^{\bt_i}\max(|x|,\,1)^{p_i'\rho/n}$, 
$\bt_i+p_i'\rho/n>-d$.
Then each $\sg_{i}$ also belongs to $A_{\8}$.
Suppose that $\rho<-d/r$.
Noticing that $\rho$ is negative,
we first observe that
\begin{align*}
\sg_i(S)
&\le
\min\lt(
\int_S|x|^{\bt_i+p_i'\rho/n}\,{\rm d}x,
\int_S|x|^{\bt_i}\,{\rm d}x
\rt)\\
&\approx
\min\lt(
\max(\ell_S,|c_S|)^{\bt_i+p_i'\rho/n}
|S|,
\max(\ell_S,|c_S|)^{\bt_i}|S|
\rt)\\
&=
|S|
\max(\ell_S,|c_S|)^{\bt_i}
\max(\ell_{S},\,|c_{S}|,\,1)^{p_i'\rho/n}.
\end{align*}
Therefore, we have
\begin{align*}
&\lt(
K(S)|S|^{n-1}
\prod_{i=1}^n
\lt(\frac{\sg_i(S)}{|S|}\rt)^{1/p_i'}
\rt)^r
|E_{\cS}(S)|
\\ &\lesssim
\lt(
\ell_{S}^{\al}
\max(\ell_{S},\,|c_{S}|)^{
\bt_1/p_1'+\cdots+\bt_n/p_n'
}\rt)^r
|E_{\cS}(S)|
\max(\ell_{S},\,|c_{S}|,\,1)^{r\rho}
\\ &\le
\lt(
\max(\ell_{S},\,|c_{S}|)^{
\al+\bt_1/p_1'+\cdots+\bt_n/p_n'
}\rt)^r
|E_{\cS}(S)|
\max(\ell_{S},\,|c_{S}|,\,1)^{r\rho}
\\ &=
|E_{\cS}(S)|
\max(\ell_{S},\,|c_{S}|,\,1)^{r\rho},
\end{align*}
by the assumption that
$\al+\bt_1/p_1'+\cdots+\bt_n/p_n'=0$.
Thus, for the finiteness, 
\begin{align*}
&\sum_{S\in\cS}
|E_{\cS}(S)|
\max(\ell_{S},\,|c_{S}|,\,1)^{r\rho}\\
&\approx
\sum_{S\in\cS,\,\overline{S}\ni0}
|E_{\cS}(S)|
\max(\ell_{S},\,1)^{r\rho}
+
\sum_{S\in\cS,\,\overline{S}\not\ni0}
|E_{\cS}(S)|
\max(|c_{S}|,\,1)^{r\rho}\\
&\lesssim
\sum_{j\in\Z}
2^{jd}
\max(2^j,1)^{r\rho}
+
\sum_{S\in\cS}
\int_{E_{\cS}(S)}\max(|x|,1)^{r\rho}\,{\rm d}x\\
&<\infty
\end{align*}
since $\rho<-d/r$.

\section{Application to the Bessel potentials}\label{sec4}
In what follows we apply Theorems 
\ref{thm2.1} and \ref{thm2.2}
to the Bessel potentials.

\subsection{Bessel potentials}\label{ssec4.1}
For a~rapidly decreasing function $f$,
define its Fourier transform and 
its inverse Fourier transform as
\[
\hat{f}(\xi)
:=
\int_{\R^d}
f(x)e^{-2\pi ix\cdot\xi}\,{\rm d}x
\quad\text{and}\quad
f^{\vee}(x)
:=
\int_{\R^d}
f(\xi)e^{2\pi ix\cdot\xi}\,{\rm d}\xi.
\]

\begin{defn}\label{def2.1}
Let $0<\al<d$ and $0<\lm<\8$. 
The \textit{Bessel potential} of order $\al$ 
is the operator 
$J_{\al,\lm}=(\lm^2I-\Dl)^{-\al/2}$ 
given by
\[
J_{\al,\lm}(f)
:=
\lt(\hat{f}\hat{G_{\al,\lm}}\rt)^{\vee}
=
f*G_{\al,\lm},
\]
where 
\[
G_{\al,\lm}(x)
=
\lt((\lm^2+4\pi^2|\xi|^2)^{-\al/2}\rt)^{\vee}(x).
\]
If $\lm=1$, 
we omit the subscript $1$ and 
simply write
$J_{\al}$ and $G_{\al}$, respectively.
\end{defn}

By the definition, one sees that 
\begin{equation}\label{4.1}
G_{\al,\lm}(x)
=
\lm^{d-\al}G_{\al}(\lm x),
\quad x\in\R^d.
\end{equation}

\begin{lem}
[{\rm\cite{HKST}}]
\label{lem4.1}
Let $0<\al<d$. Then 
$G_{\al}$ is a~smooth function on 
$\R^d\setminus\{0\}$ 
that satisfies 
$G_{\al}(x)>0$, $x\in\R^d$, and 
there exist positive finite constants 
$C_{\al,d}$ and $c_{\al,d}$ 
such that
\[
G_{\al}(x)
\le
\begin{cases}
C_{\al,d}e^{-\frac{|x|}2},
\quad\text{when $|x|>1$},
\\
c_{\al,d}|x|^{\al-d},
\quad\text{when $|x|\le 1$}.
\end{cases}
\]
\end{lem}

For $\tau\in\{0,\pm\frac13\}^d$, 
we define the dyadic grid by 
\[
\cD^{\tau}
:=
\{2^{-k}(m+(-1)^k\tau+[0,1)^d):\,
k\in\Z,\,m\in\Z^d\}.
\]

\begin{lem}
[{\rm\cite{HKST}}]
\label{lem4.2}
Let $0<\al<d$ and $0<\lm<\8$. 
We have that, 
for some appropriate sparse families 
$\cS_{\tau}\subset\cD^{\tau}$, 
$\tau\in\{0,\pm\frac13\}^d$, 
\[
J_{\al,\lm}(f)
\lesssim
\sum_{\tau\in\{0,\pm\frac13\}^d}
S_{\al,\lm}^{\tau}f,
\qquad f\ge 0.
\]
Here, 
\begin{equation}\label{4.2}
S_{\al,\lm}^{\tau}f
:=
\sum_{S\in\cS_{\tau}}
\frac{\min((\lm\ell_{S})^{\al},\,1)}{\lm^{\al}|S|}
\int_{S}f\,{\rm d}x\,\1_{S},
\qquad f\ge 0.
\end{equation}
\end{lem}

\subsection{
Proof of Theorems \ref{thm4.1} and \ref{thm4.2}, 
for preparation
}\label{ssec4.2}
We treat only the slightly more complicated bump case
of Theorem \ref{thm4.1} part (b) 
with $1/p_1+\cdots+1/p_n=1$.
The other cases are quite similar.

With the motivation to \eqref{4.2}, 
we set 
\[
K(S)
:=
\frac{\min((\lm\ell_{S})^{\al},\,1)}{\lm^{\al}|S|^{n-1}},
\qquad 0<\al<(n-1)d,
\quad 0<\lm<\8.
\]
We let 
\[
A_0'(\lm)
:=
\max\lt(\begin{array}{l}
\ds
\sup_{S\in\cS:\,
\ell_{S}\le 1/\lm}
\ell_{S}^{\al}\prod_{i=1}^n
\lt(\frac{\sg_i(S)}{|S|}\rt)^{\frac1{p_i'}},\\
\ds
\lm^{-\al}
\sup_{S\in\cS:\,
\ell_{S}>1/\lm}
\prod_{i=1}^n
\lt(\frac{\sg_i(S)}{|S|}\rt)^{\frac1{p_i'}}
\end{array}\rt)
\]
and let 
\[
A_0(\lm)
:=
\max\lt(\begin{array}{l}
\ds
\sup_{S\in\cS:\,
\ell_{S}\le 1/\lm}
\ell_{S}^{\al}\prod_{i=1}^n
\lt(\frac{\sg_i^{\tht}(S)}{|S|}\rt)^{\frac1{\tht p_i'}},\\
\ds
\lm^{-\al}
\sup_{S\in\cS:\,
\ell_{S}>1/\lm}
\prod_{i=1}^n
\lt(\frac{\sg_i^{\tht}(S)}{|S|}\rt)^{\frac1{\tht p_i'}}
\end{array}\rt).
\]

\begin{prp}\label{prp4.1}
\begin{itemize}
\item[{\rm(a)}] 
Assume that
\begin{equation*}
\max\lt(\begin{array}{l}
\ds
\sup_{S\in\cS:\,
\ell_{S}\le 1}
\ell_{S}^{\al}\prod_{i=1}^n
\lt(\frac{\sg_i(S)}{|S|}\rt)^{\frac1{p_i'}},\\
\ds
\sup_{S\in\cS:\,
\ell_{S}>1}
\prod_{i=1}^n
\lt(\frac{\sg_i(S)}{|S|}\rt)^{\frac1{p_i'}}
\end{array}\rt)
=C_0'<\8
\end{equation*}
and that
\begin{equation*}
\lim_{\lm\to\8}
\sup_{S\in\cS:\,
\ell_{S}\le 1/\lm}
\ell_{S}^{\al}\prod_{i=1}^n
\lt(\frac{\sg_i(S)}{|S|}\rt)^{\frac1{p_i'}}
=0.
\end{equation*}
Then, for any $\eps>0$, 
there exists $\lm_0>0$ such that 
\begin{equation*}
A_0'(\lm_0)<\eps.
\end{equation*}
\item[{\rm(b)}] 
Assume that
\begin{equation}\label{4.3}
\max\lt(\begin{array}{l}
\ds
\sup_{S\in\cS:\,
\ell_{S}\le 1}
\ell_{S}^{\al}\prod_{i=1}^n
\lt(\frac{\sg_i^{\tht}(S)}{|S|}\rt)^{\frac1{\tht p_i'}},\\
\ds
\sup_{S\in\cS:\,
\ell_{S}>1}
\prod_{i=1}^n
\lt(\frac{\sg_i^{\tht}(S)}{|S|}\rt)^{\frac1{\tht p_i'}}
\end{array}\rt)
=C_0<\8
\end{equation}
and that
\begin{equation}\label{4.4}
\lim_{\lm\to\8}
\sup_{S\in\cS:\,
\ell_{S}\le 1/\lm}
\ell_{S}^{\al}\prod_{i=1}^n
\lt(\frac{\sg_i^{\tht}(S)}{|S|}\rt)^{\frac1{\tht p_i'}}
=0.
\end{equation}
Then, for any $\eps>0$, 
there exists $\lm_0>0$ such that 
\begin{equation}\label{4.5}
A_0(\lm_0)<\eps.
\end{equation}
\end{itemize}
\end{prp}

\begin{proof}
We only verify (b).
Take an $\eps>0$. 
First we choose an integer 
$N_1\in\N$ so that 
\begin{equation}\label{4.6}
\frac{C_0}{(2^{N_1})^{\al}}<\eps.
\end{equation}
Next, thanks to \eqref{4.4}, 
we choose an integer $N_0>N_1$ 
so that 
\begin{equation}\label{4.7}
\sup_{S\in\cS:\,
\ell_{S}\le 2^{N_1-N_0}}
\ell_{S}^{\al}\prod_{i=1}^n
\lt(\frac{\sg_i^{\tht}(S)}{|S|}\rt)^{\frac1{\tht p_i'}}
<\eps.
\end{equation}

Let $\lm_0=2^{N_0}$.
By 
$1/\lm_0=2^{-N_0}<2^{N_1-N_0}$
and \eqref{4.7},
\[
\sup_{S\in\cS:\,
\ell_{S}\le 1/\lm_0}
\ell_{S}^{\al}\prod_{i=1}^n
\lt(\frac{\sg_i^{\tht}(S)}{|S|}\rt)^{\frac1{\tht p_i'}}
<\eps.
\]
By $N_0>N_1$, 
\eqref{4.3} and \eqref{4.6},
\[
\frac1{\lm_0^{\al}}
\sup_{S\in\cS:\,
\ell_{S}>1}\prod_{i=1}^n
\lt(\frac{\sg_i^{\tht}(S)}{|S|}\rt)^{\frac1{\tht p_i'}}
<\eps.
\]
By \eqref{4.7},
\begin{align*}
&\frac1{\lm_0^{\al}}
\sup\lt\{\prod_{i=1}^n
\lt(\frac{\sg_i^{\tht}(S)}{|S|}\rt)^{\frac1{\tht p_i'}}
:\,S\in\cS,\,
2^{-N_0}<\ell_{S}\le 2^{N_1-N_0}
\rt\}
\\ &=
\sup\lt\{
\frac{\ell_{S}^{\al}}{(\lm_0\ell_{S})^{\al}}
\prod_{i=1}^n
\lt(\frac{\sg_i^{\tht}(S)}{|S|}\rt)^{\frac1{\tht p_i'}}
:\,S\in\cS,\,
2^{-N_0}<\ell_{S}\le 2^{N_1-N_0}
\rt\}
\\ &\le
\sup\lt\{
\ell_{S}^{\al}\prod_{i=1}^n
\lt(\frac{\sg_i^{\tht}(S)}{|S|}\rt)^{\frac1{\tht p_i'}}
:\,S\in\cS,\,
\ell_{S}\le 2^{N_1-N_0}
\rt\}<\eps.
\end{align*}
By \eqref{4.3} and \eqref{4.6},
\begin{align*}
&\frac1{\lm_0^{\al}}
\sup\lt\{\prod_{i=1}^n
\lt(\frac{\sg_i^{\tht}(S)}{|S|}\rt)^{\frac1{\tht p_i'}}
:\,S\in\cS,\,
2^{N_1-N_0}<\ell_{S}\le 1
\rt\}
\\ &=
\sup\lt\{
\frac{\ell_{S}^{\al}}{(\lm_0\ell_{S})^{\al}}
\prod_{i=1}^n
\lt(\frac{\sg_i^{\tht}(S)}{|S|}\rt)^{\frac1{\tht p_i'}}
:\,S\in\cS,\,
2^{N_1-N_0}<\ell_{S}\le 1
\rt\}
\\ &\le
\frac1{(2^{N_1})^{\al}}
\sup\lt\{
\ell_{S}^{\al}\prod_{i=1}^n
\lt(\frac{\sg_i^{\tht}(S)}{|S|}\rt)^{\frac1{\tht p_i'}}
:\,S\in\cS,\,\ell_{S}\le 1
\rt\}<\eps.
\end{align*}
These prove \eqref{4.5}.
\end{proof}

\subsection{
Proof of Theorems \ref{thm4.1} and \ref{thm4.2}
}\label{ssec4.3}
For $0<\al<d$ and $0<\lm<\8$, 
the identity $I$ can be decomposed as
\begin{align*}
I
&=
(\lm^2I-\Dl)^{-\al/2}
\circ
(\lm^2I-\Dl)^{\al/2}
(\lm^{\al}I+(-\Dl)^{\al/2})^{-1}
\circ
(\lm^{\al}I+(-\Dl)^{\al/2})
\\ &=:
J_{\al,\lm}
\circ
T_m
\circ
(\lm^{\al}I+(-\Dl)^{\al/2}).
\end{align*}
Using Mihlin's multiplier theorem 
(see, e.g. \cite[Theorem 5.2.7]{Gr}), 
we have that $T_m$ is bounded on $L^p({\rm d}x)$.
Thus, we need only verify that, 
for any $\eps>0$, 
there exists $\lm_0>0$ 
such that the operator norm
\begin{equation}\label{4.8}
\|J_{\al,\lm_0}\|_{L^p({\rm d}x) \to L^q(v)}<\eps.
\end{equation}

\vspace{5mm}

For Theorem \ref{thm4.1},
this fact \eqref{4.8} follows 
from Lemma \ref{lem4.2} and
Proposition \ref{prp4.1},
with necessary modification,
by setting
$n=2$,
$\sg_1={\rm d}x$,
$\sg_2=v$,
$p_1=p$ and 
$p_2=q'$.

For Theorem \ref{thm4.1} part (a),
we must verify that
\[
\max\lt(\begin{array}{l}
\ds
\sup_{Q\in\cQ(\R^d):\,
\ell_{Q}\le 1}
|Q|^{\al/d-1/p}v(Q)^{1/q},\\
\ds
\sup_{Q\in\cQ(\R^d):\,\ell_{Q}>1}
\frac{v(Q)^{1/q}}{|Q|^{1/p}}
\end{array}\rt)
<\8
\]
and that
\[
\lim_{\lm\to\8}
\sup_{Q\in\cQ(\R^d):\,
\ell_{Q}\le 1/\lm}
|Q|^{\al/d-1/p}v(Q)^{1/q}
=0.
\]
By invoking our assumptions, 
we have to prove
\[
{\rm(I)}
:=
\sup_{Q\in\cQ(\R^d):\,
\ell_{Q}>1}
\frac{v(Q)^{1/q}}{|Q|^{1/p}}
<\8.
\]
Indeed, since
\[
\sup_{Q\in\cQ(\R^d):\,
\ell_{Q}=1}
v(Q)^{1/q}
\le
\sup_{Q\in\cQ(\R^d):\,
\ell_{Q}\le 1}
|Q|^{\al/d-1/p}v(Q)^{1/q}
=a_1<\8,
\]
we have that
\[
{\rm(I)}
\lesssim
\sup_{Q\in\cQ(\R^d):\,
\ell_{Q}>1}
a_1\#(\Z^d\cap Q)^{1/q-1/p}
\le a_1<\8
\]
from the covering $Q\subset\bigcup_{m\in\Z^d\cap Q}(m+[-1,1)^d)$ with at most $2^d$-fold overlap.

For Theorem \ref{thm4.1} part (b),
we must verify that
\[
\max\lt(\begin{array}{l}
\ds
\sup_{Q\in\cQ(\R^d):\,
\ell_{Q}\le 1}
|Q|^{\al/d+1/q-1/p}
\lt(\frac{v^{\tht}(Q)}{|Q|}\rt)^{\frac1{\tht q}},\\ 
\ds
\sup_{Q\in\cQ(\R^d):\,
\ell_{Q}>1}
|Q|^{1/q-1/p}
\lt(\frac{v^{\tht}(Q)}{|Q|}\rt)^{\frac1{\tht q}}
\end{array}\rt)
<\8
\]
and that
\[
\lim_{\lm\to\8}
\sup_{Q\in\cQ(\R^d):\,
\ell_{Q}\le 1/\lm}
|Q|^{\al/d+1/q-1/p}
\lt(\frac{v^{\tht}(Q)}{|Q|}\rt)^{\frac1{\tht q}}
=0.
\]
By invoking our assumptions, 
we have to prove
\[
{\rm(II)}
:=
\sup_{Q\in\cQ(\R^d):\,
\ell_{Q}>1}
|Q|^{1/q-1/p}
\lt(\frac{v^{\tht}(Q)}{|Q|}\rt)^{\frac1{\tht q}}
<\8.
\]
Indeed, since
\[
\sup_{Q\in\cQ(\R^d):\,
\ell_{Q}=1}
\lt(v^{\tht}(Q)\rt)^{\frac1{\tht q}}
\le
\sup_{Q\in\cQ(\R^d):\,
\ell_{Q}\le 1}
|Q|^{\al/d+1/q-1/p}
\lt(\frac{v^{\tht}(Q)}{|Q|}\rt)^{\frac1{\tht q}}
=a_2<\8,
\]
we have that
\[
{\rm(II)}
\lesssim
\sup_{Q\in\cQ(\R^d):\,
\ell_{Q}>1}
|Q|^{1/q-1/p}
\lt(\frac{a_2^{\tht q}\#(\Z^d\cap Q)}{|Q|}\rt)^{\frac1{\tht q}}
\lesssim a_2<\8
\]
from the covering $Q\subset\bigcup_{m\in\Z^d\cap Q}(m+[-1,1)^d)$ with at most $2^d$-fold overlap.

\vspace{5mm}

For Theorem \ref{thm4.2},
the fact \eqref{4.8} follows 
from the following.
We need the following facts, 
see Theorem \ref{thm2.2}.

Setting
\[
K(S)
:=
\frac{\min((\lm\ell_{S})^{\al},\,1)}{\lm^{\al}|S|},
\qquad 0<\al<d,
\quad 0<\lm<\8,
\]
we let, for Theorem \ref{thm4.2} part (a),
\begin{align*}
A_0
&=
\lt[
\sum_{S\in\cS}
\lt(
K(S)|S|
\lt(\frac{v(S)}{|S|}\rt)^{1/q}
\rt)^r
|E_{\cS}(S)|
\rt]^{1/r}
\\ &=
\lt[
\sum_{S\in\cS}
\lt(
K(S)|S|
\lt(\frac{v(S)}{|S|}\rt)^{
\frac1{\tht_2 p}
+
\frac{\tht_1}r
}
\rt)^r
|E_{\cS}(S)|
\rt]^{1/r}
\\ &\le
\lt[
\sum_{S\in\cS}
\lt(\frac{v(S)}{|S|}\rt)^{\tht_1}
|E_{\cS}(S)|
\rt]^{1/r}
\cdot
\sup_{S\in\cS}
K(S)|S|
\lt(\frac{v(S)}{|S|}\rt)^{\frac1{\tht_2 p}}
\\ &\lesssim
\|v\|_{L^{\tht_1}({\rm d}x)}^{\tht_1/r}
\cdot
\sup_{S\in\cS}
K(S)|S|
\lt(\frac{v(S)}{|S|}\rt)^{\frac1{\tht_2 p}}.
\end{align*}

Similarly,
for Theorem \ref{thm4.2} part (b),
\begin{align*}
A_0
&=
\lt[
\sum_{S\in\cS}
\lt(
K(S)|S|
\lt(\frac{v^{\tht}(S)}{|S|}\rt)^{\frac1{\tht q}}
\rt)^r
|E_{\cS}(S)|
\rt]^{1/r}
\\ &=
\lt[
\sum_{S\in\cS}
\lt(
K(S)|S|
\lt(\frac{v^{\tht}(S)}{|S|}\rt)^{
\frac1{\tht\tht_2 p}
+
\frac{\tht_1}{\tht r}
}
\rt)^r
|E_{\cS}(S)|
\rt]^{1/r}
\\ &\le
\lt[
\sum_{S\in\cS}
\lt(\frac{v^{\tht}(S)}{|S|}\rt)^{\tht_1/\tht}
|E_{\cS}(S)|
\rt]^{1/r}
\cdot
\sup_{S\in\cS}
K(S)|S|
\lt(\frac{v^{\tht}(S)}{|S|}\rt)^{\frac1{\tht\tht_2 p}}
\\ &\lesssim
\|v\|_{L^{\tht_1}({\rm d}x)}^{\tht_1/r}
\cdot
\sup_{S\in\cS}
K(S)|S|
\lt(\frac{v^{\tht}(S)}{|S|}\rt)^{\frac1{\tht\tht_2 p}}.
\end{align*}
Furthermore,
\[
\sup_{Q\in\cQ(\R^d):\,
\ell_{Q}>1}
\lt(\frac{v(Q)}{|Q|}\rt)^{\frac1{\tht_2 p}}
\le
\|v\|_{L^{\tht_1}({\rm d}x)}^{
\frac1{\tht_2 p}
}<\8
\]
and
\[
\sup_{Q\in\cQ(\R^d):\,
\ell_{Q}>1}
\lt(\frac{v^{\tht}(Q)}{|Q|}\rt)^{\frac1{\tht\tht_2 p}}
\le
\|v\|_{L^{\tht_1}({\rm d}x)}^{
\frac1{\tht_2 p}
}<\8.
\]

These complete the proof of 
Theorems \ref{thm4.1} and \ref{thm4.2}.
\qed

\section{Appendix}\label{sec5}
In what follows, 
as an appendix, 
we present the proof of the so-called 
\textit{$A_2$ conjecture}, 
in order to illustrate the overall strategy employed throughout this paper, 
following the approach of \cite{Le} and \cite{Mo}.
For further details on the $A_2$ conjecture, 
we refer the reader to the excellent survey \cite{Hy} 
and the paper \cite{Le}.

Let $\cS\subset\cD$ be an~$\eta$-sparse family.
Define
\[
\cA_{\cS}f
:=
\sum_{S\in\cS}
\frac1{|S|}\int_{S}f\,{\rm d}x\,\1_{S}.
\]
Suppose that $w\in A_p$.
We shall prove that
\[
\|\cA_{\cS}f\|_{L^p(w)}
\le c_{n,p,\eta}
[w]_{A_p}^{\max(1,p'/p)}
\|f\|_{L^p(w)}.
\]

If we set
$\sg=w^{-\frac1{p-1}}=w^{-p'/p}$,
then $\sg\in A_{p'}$
with the characteristic relation
\[
[\sg]_{A_{p'}} = [w]_{A_p}^{\frac{1}{p-1}} = [w]_{A_p}^{p'/p}.
\]
By the Sawyer trick
(Notice that
if we let $F=f\sg$, then 
$f^p\sg=F^pw$),
we need only verify that
\[
\|\cA_{\cS}[f\sg]\|_{L^p(w)}
\le c_{n,p,\eta}
[w]_{A_p}^{\max(1,p'/p)}
\|f\|_{L^p(\sg)}.
\]
By duality, this is equivalent to
\[
\|\cA_{\cS}[gw]\|_{L^{p'}(\sg)}
\le c_{n,p,\eta}
[w]_{A_p}^{\max(1,p'/p)}
\|g\|_{L^{p'}(w)},
\]
and, these can be verified by the following bilinear embedding inequality:
\[
\sum_{S\in\cS}
\frac1{|S|}
\int_{S}f\sg\,{\rm d}x
\int_{S}gw\,{\rm d}x
\le c_{n,p,\eta}
[w]_{A_p}^{\max(1,p'/p)}
\|f\|_{L^p(\sg)}
\|g\|_{L^{p'}(w)}.
\]

There holds
\begin{align*}
{\rm(i)}
&:=
\sum_{S\in\cS}
\frac1{|S|}
\int_{S}f\sg\,{\rm d}x
\int_{S}gw\,{\rm d}x
\\ &=
\sum_{S\in\cS}
\frac1{|S|}
\frac{\sg(S)}{\sg(E_{\cS}(S))^{1/p}}
\frac{w(S)}{w(E_{\cS}(S))^{1/p'}}
\\ &\quad\times
\fint_{S}f\,{\rm d}\sg
\cdot
\sg(E_{\cS}(S))^{1/p}
\fint_{S}g\,{\rm d}w
\cdot
w(E_{\cS}(S))^{1/p'}.
\end{align*}
By letting
\[
T_p(w;S)
:=
\frac1{|S|}
\frac{\sg(S)}{\sg(E_{\cS}(S))^{1/p}}
\frac{w(S)}{w(E_{\cS}(S))^{1/p'}},
\]
and using H\"{o}lder's inequality,
we have that
\begin{align*}
{\rm(i)}
&\le
\sup_{S\in\cS}T_p(w;S)
\sum_{S\in\cS}
\fint_{S}f\,{\rm d}\sg
\cdot
\sg(E_{\cS}(S))^{1/p}
\fint_{S}g\,{\rm d}w
\cdot
w(E_{\cS}(S))^{1/p'}
\\ &\le
\sup_{S\in\cS}T_p(w;S)
\times
\lt[
\sum_{S\in\cS}
\lt(\fint_{S}f\,{\rm d}\sg\rt)^p
\sg(E_{\cS}(S))
\rt]^{1/p}
\lt[
\sum_{S\in\cS}
\lt(\fint_{S}g\,{\rm d}w\rt)^{p'}
w(E_{\cS}(S))
\rt]^{1/p'} \\
&\le
\sup_{S\in\cS}T_p(w;S)
\lt(
\int_{\R^d}
M_{\sg}f(x)^p\sg(x)\,{\rm d}x
\rt)^{1/p}
\lt(
\int_{\R^d}
M_wg(x)^{p'}w(x)\,{\rm d}x\rt)^{1/p'}
\\ &\lesssim
\sup_{S\in\cS}T_p(w;S)
\|f\|_{L^p(\sg)}
\|g\|_{L^{p'}(w)},
\end{align*}
where we have used the boundedness of 
the weighted dyadic Hardy-Littlewood  maximal operators 
$M_{\sg}$ and $M_w$ on 
$L^{p}(\sg)$ and $L^{p'}(w)$, 
respectively.

It remains to show that  
\[
T_p(w;S)
=
\frac1{|S|}
\frac{\sg(S)}{\sg(E_{\cS}(S))^{1/p}}
\frac{w(S)}{w(E_{\cS}(S))^{1/p'}}
\le c_{n,p,\eta}
\lt([w]_{A_p}^{1/p}\rt)^{\max(p,p')}.
\]
By H\"{o}lder's inequality,
\[
\eta|S|
\le
|E_{\cS}(S)|
\le
w(E_{\cS}(S))^{1/p}
\sg(E_{\cS}(S))^{1/p'}.
\]
From this,
\[
\frac{w(S)^{1/p}\sg(S)^{1/p'}}{w(E_{\cS}(S))^{1/p}\sg(E_{\cS}(S))^{1/p'}}
\le
\frac{w(S)^{1/p}\sg(S)^{1/p'}}{\eta|S|}
\le
\frac1\eta
[w]_{A_p}^{1/p},
\]
and therefore,
\begin{align*}
T_p(w;S)
&=
\frac{w(S)^{1/p}\sg(S)^{1/p'}}{|S|}
\frac{w(S)^{1/p'}}{w(E_{\cS}(S))^{1/p'}}
\frac{\sg(S)^{1/p}}{\sg(E_{\cS}(S))^{1/p}}
\\ &\le
[w]_{A_p}^{1/p}
\lt[\lt(\frac{w(S)}{w(E_{\cS}(S))}\rt)^{1/p}\rt]^{p/p'}
\lt[\lt(\frac{\sg(S)}{\sg(E_{\cS}(S))}\rt)^{1/p'}\rt]^{p'/p}
\\ &\le
[w]_{A_p}^{1/p}
\lt(
\frac{w(S)^{1/p}\sg(S)^{1/p'}}{w(E_{\cS}(S))^{1/p}\sg(E_{\cS}(S))^{1/p'}}
\rt)^{\max(p/p',p'/p)}
\\ &\le
c_{n,p,\eta}
\lt([w]_{A_p}^{1/p}\rt)^{
1+\max(p/p',p'/p)
}
\\ &=
c_{n,p,\eta}
\lt([w]_{A_p}^{1/p}\rt)^{\max(p,p')}
=
c_{n,p,\eta}
[w]_{A_p}^{\max(1,p'/p)}.
\end{align*}
This completes the proof.
\qed

\end{document}